\documentclass[11pt,a4paper,reqno]{amsart}
\pdfoutput=1
\usepackage{amssymb,interval,eulervm,palatino,fullpage}
\usepackage[USenglish]{babel}

\allowdisplaybreaks[4]
\makeatletter
\@namedef{subjclassname@2020}{%
  \textup{2020} Mathematics Subject Classification}
\makeatother

\newtheorem{thm}{Theorem}
\newtheorem{lemma}[thm]{Lemma}
\newtheorem{prop}[thm]{Proposition}
\newtheorem{df}[thm]{Definition}

\newcommand{\A}{\mathcal{A}}
\newcommand{\HH}{\mathcal{H}}
\newcommand{\C}{\mathbb{C}}
\newcommand{\N}{\mathbb{N}}
\newcommand{\R}{\mathbb{R}}
\newcommand{\inner}[1]{\left<#1\right>}
\newcommand{\ket}[1]{\left|#1\right>}
\newcommand{\mv}[1]{\smash[b]{\underline{#1}}}

\begin{document}

\title{On irreducible representations of a class of quantum spheres}

\date{May 2022}

\author[F.~D'Andrea and G.~Landi]{Francesco D'Andrea and Giovanni Landi}

\address[F.~D'Andrea]
{Universit\`a di Napoli ``Federico II'' 
\newline \indent
and I.N.F.N. Sezione di Napoli, Complesso MSA, Via Cintia, 80126 Napoli, Italy}
\email{francesco.dandrea@unina.it}

\address[G.~Landi]
{Universit\`a di Trieste,
Via A. Valerio, 12/1, 34127  Trieste, Italy
\newline \indent
Institute for Geometry and Physics (IGAP) Trieste, Italy 
\newline \indent and INFN, Trieste, Italy}
\email{landi@units.it}



\maketitle

\section{Introduction}
 
Throughout the paper we shall let \mbox{$0<q<1$} be a deformation parameter and $n$ be a positive integer.  
We denote by $\A(S^{4n-1}_q)$ the complex unital $*$-algebra generated by elements $\{x_i,y_i\}_{i=1}^n$ and their adjoints, subject to the relations in Definition~\ref{de:sphere} below. This sphere is a comodule algebra for the quantum symplectic group $\A(Sp_q(n))$, with coaction $\A(S^{4n-1}_q) \to \A(Sp_q(n)) \otimes \A(S^{4n-1}_q)$.
In fact this is a quantum homogeneous space and the algebra $\A(S^{4n-1}_q)$ sits as a subalgebra of the algebra $\A(Sp_q(n))$. Representations of $\A(S^{4n-1}_q)$ can be obtained as restrictions of representations of $\A(Sp_q(n))$, see e.g.\ \cite{Sau17}. 

Let $\A(\Sigma^{2n+1}_q)$ be the quotient of $\A(S^{4n-1}_q)$ by the two-sided $*$-ideal generated by
the elements $\{x_i\}_{i=1}^{n-1}$. As customary, we interpret a quotient algebra as consisting of ``functions'' on a quantum subspace and think of $\Sigma^{2n+1}_q$ as a quantum subsphere of $S^{4n-1}_q$. 
If we further quotient by the ideal generated by the coordinate $x_n$, we get a $(2n-1)$-dimensional Vaksman-Soibelman quantum sphere \cite{VS91}, whose representation theory is well known (see e.g.~\cite{HL04}).
In this short letter we give an independent derivation of the bounded irreducible $*$-representations of the algebra $\A(\Sigma^{2n+1}_q)$ that do not annihilate the generator $x_n$.

\begin{df}\label{de:sphere}
We denote by $\A(S^{4n-1}_q)$ the complex unital $*$-algebra generated by elements $\{x_i,y_i\}_{i=1}^n$ and their adjoints, subject to the following relations. Firstly, one has:
\begin{gather}
x_ix_j=q^{-1}x_jx_i \;\;(i<j) ,\qquad
y_iy_j=q^{-1}y_jy_i \;\;(i>j) ,\qquad
x_iy_j=q^{-1}y_jx_i \;\;(i\neq j) , \label{eq:useD} \\
y_ix_i=q^2x_iy_i+(q^2-1)\sum_{k=1}^{i-1}q^{i-k}x_ky_k , \label{eq:useC}
\end{gather}
Next, one has
\begin{align}
x_ix_i^* &=x_i^*x_i+(1-q^2)\sum_{k=1}^{i-1}x_k^*x_k \label{eq:useA} \\
y_iy_i^* &=y_i^*y_i+(1-q^2)\left\{q^{2(n+1-i)}x_i^*x_i+\sum_{k=1}^nx_k^*x_k+\sum_{k=i+1}^{n}y_k^*y_k\right\} \label{eq:useB} \\
x_iy_i^* &=q^2y_i^*x_i \label{eq:useH} \\
x_ix_j^* &=qx_j^*x_i
\qquad(i\neq j) \label{eq:useG} \\
y_iy_j^* &=qy_j^*y_i-(q^2-1)q^{2n+2-i-j}x_i^*x_j
\qquad(i\neq j) \label{eq:useF} \\
x_iy_j^* &=qy_j^*x_i
\qquad(i<j)  \label{eq:useE} \\
x_iy_j^* &=qy_j^*x_i+(q^2-1)q^{i-j}y_i^*x_j
\qquad(i>j) \label{eq:useI}
\end{align}
Finally, one has the sphere relation:
\begin{equation}\label{eq:sphererel}
\sum_{i=1}^n(x_i^*x_i+y_i^*y_i)=1 \, .
\end{equation}
\end{df}
\noindent
One passes to the notations of \cite{LPR06} by setting $y_i=x_{2n+1-i}$ and replacing $q$ by $q^{-1}$.

Let $\A(\Sigma^{2n+1}_q)$ be the quotient of $\A(S^{4n-1}_q)$ by the two-sided $*$-ideal generated by
$\{x_i\}_{i=1}^{n-1}$. Let us write down explicitly the relations in this quotient algebra.
If we rename $y_{n+1}:=x_n$, it follows from \eqref{eq:useA} that $y_{n+1}$ is normal.
The remaining relations become:
\begin{align}
y_iy_j &=q^{-1}y_jy_i \qquad\qquad \big(\;i>j \wedge (i,j) \neq (n+1,n) \;\big)  \label{eq:1}  \\
y_i^*y_j &=q^{-1}y_jy_i^* \qquad\qquad \big(\;i>j \wedge (i,j) \neq (n+1,n)\;\big) \label{eq:2}  \\
y_{n+1}y_n &=q^{-2}y_ny_{n+1} \qquad\quad
y_{n+1}^*y_n =q^{-2}y_ny_{n+1}^* \label{eq:3}   \\
[y_i,y_i^*] &=(1-q^2) \sum_{k=i+1}^{n+1}y_k^*y_k\ \qquad ( i\neq n ) \label{eq:subA} \\
[y_n,y_n^*] &=(1-q^4)y_{n+1}^*y_{n+1} \label{eq:subB} 
\end{align}
plus the ones obtained by adjunction and the sphere relation:
\begin{equation}\label{eq:spherebis} 
\sum_{i=1}^{n+1}y_i^*y_i=1 \, .
\end{equation}
Using these relations, it is straightforward to check the following statement.

\begin{prop}\label{prop:rep}
For every $\lambda\in U(1)$, an irreducible bounded $*$-representation $\pi_\lambda$ of $\mathcal{A}(\Sigma^{2n+1}_q)$
on $\ell^2(\N^n)$ is given by the formulas:
\begin{align*}
\pi_\lambda(y_i)\ket{\mv{k}} &=q^{k_1+\ldots+k_{i-1}}\sqrt{1-q^{2k_i}}\ket{\mv{k}-\mv{e}_i} \qquad (1\leq i\leq n-1)\\
\pi_\lambda(y_n)\ket{\mv{k}} &=q^{k_1+\ldots+k_{n-1}}\sqrt{1-q^{4k_n}}\ket{\mv{k}-\mv{e}_n} \;,\\
\pi_\lambda(y_{n+1})\ket{\mv{k}} &=\lambda q^{|\mv{k}|+k_n}\ket{\mv{k}} \;,
\end{align*}
where $\mv{k}=(k_1,\ldots,k_n)\in\N^n$, $|\mv{k}|:=k_1+\ldots+k_n$,
$\{\ket{\mv{k}}\}_{\mv{k}\in\N^n}$ is the canonical orthonormal basis of $\ell^2(\N^n)$ and $\mv{e}_i$ the $i$-th row of the identity matrix of order $n$.
\end{prop}
Our aim now is to prove the next proposition.

\begin{prop}\label{thm:two}
Any irreducible bounded $*$-representation of $\A(\Sigma^{2n+1}_q)$ that does not annihilate $x_n$ is unitarily equivalent to one of the representations in Proposition~\ref{prop:rep}.
\end{prop}

We need a few preliminary lemmas.

\begin{lemma}\label{lemma:aux}
For all $m\geq 1$ and all $1\leq i<n$:
\begin{align*}
y_i(y_i^*)^m &=q^{2m}(y_i^*)^my_i+(1-q^{2m})(y_i^*)^{m-1}\left(1-\sum\nolimits_{k<i}y_k^*y_k\right) \\
y_n(y_n^*)^m &=q^{4m}(y_n^*)^my_n+(1-q^{4m})(y_n^*)^{m-1}\left(1-\sum\nolimits_{k<n}y_k^*y_k \right)
\end{align*}
\end{lemma}

\begin{proof}
When $m=1$, these follow from \eqref{eq:subA} and \eqref{eq:subB}, and can be rewritten using \eqref{eq:spherebis} as:
\begin{align*}
y_iy_i^* &=q^{2}y_i^*y_i+(1-q^{2})\left(1-\sum\nolimits_{k<i}y_k^*y_k\right) \quad\text{if }i<n \\
y_ny_n^* &=q^{4}y_n^*y_n+(1-q^{4})\left(1-\sum\nolimits_{k<n}y_k^*y_k \right)
\end{align*}
The general result easily follows using the latter relations, induction on $m$ and the fact that $y_k^*y_k$ commutes
with $y_i$ for all $k<i\leq n$.
\end{proof}
\begin{lemma}\label{lemma:main}
Let $A\geq 0$ and $B$ be bounded operators on a Hilbert space $\HH$ satisfying
\begin{equation}\label{eq:defining}
[B,B^*]=(1-\mu)A \qquad\quad
A+B^*B=1
\end{equation}
with $0<\mu<1$. Then, $\ker(B)=\{0\}$ if and only if $A=0$.
\end{lemma}

\begin{proof}
If $A=0$, from \eqref{eq:defining} it follows that $B$ is unitary, so that $\ker(B)=\{0\}$. We have to prove the opposite implication.
From \eqref{eq:defining} we deduce that $\|A\|\leq 1$ and:
\begin{equation}\label{eq:othersphere}
BB^*=B^*B+(1-\mu)A=1-\mu A.
\end{equation}
Since $\|\mu A\|<1$, the operator $BB^*$ has bounded inverse. Therefore
$U:=(BB^*)^{-1/2}B$ is well defined. Notice that $UU^*=1$ and $\ker(B)=\ker(U)$. From \eqref{eq:othersphere}, it follows that
$$
\mu A=1-BB^*=U(1-B^*B)U^*=UAU^* \;.
$$
Assume that $\ker(U)=\{0\}$. The identity $U(1-U^*U)=0$ implies $U(1-U^*U)v=0$ and hence $(1-U^*U)v=0$ for all $v\in\HH$.
Therefore $U^*U=1$ and $U$ is a unitary operator.

Let $\lambda\in\C$ and suppose $A-\mu^{-1}\lambda$ has bounded inverse. Then
$$
A-\lambda=\mu U^*(A-\mu^{-1}\lambda)U
$$
has bounded inverse as well. Thus, $\lambda\in\sigma(A)$ implies that $\mu^{-1}\lambda\in\sigma(A)$ and hence, by induction, that
$\mu^{-k}\lambda\in\sigma(A)$ for all $k\geq 0$. Since $A$ is bounded and the sequence $(\mu^{-k}\lambda)_{k\geq 0}$ is divergent when $\lambda\neq 0$, it follows that $\sigma(A)=\{0\}$. Hence $A=0$.
\end{proof}

\begin{lemma}\label{lemma:atleast}
Let $\pi$ be an irreducible bounded $*$-representation of $\mathcal{A}(\Sigma^{2n+1}_q)$ with $\pi(y_{n+1})\neq 0$. Then:
\begin{itemize}\itemsep=2pt
\item[(i)] $\pi(y_{n+1})$ is injective;
\item[(ii)] there exists a vector $\xi\neq 0$ such that $\pi(y_i)\xi=0$ for all $i\neq n+1$.
\end{itemize}
\end{lemma}

\begin{proof}
(i) If $a$ is any generator other than $y_{n+1}$, since $y_{n+1}a$ is a scalar multiple of $ay_{n+1}$, the operator $\pi(a)$ maps the kernel of $\pi(y_{n+1})$ to itself. Hence, $\ker\pi(y_{n+1})$ carries a subrepresentation of the irreducible representation $\pi$, so that either $\ker\pi(y_{n+1})=\{0\}$ or $\ker\pi(y_{n+1})=\HH$. The latter implies $\pi(y_{n+1})=0$, contradicting the hypothesis, so that the former must hold.

\medskip

\noindent
(ii) Given $1\leq k\leq n$, let $\HH_k:=\bigcap_{i=1}^k\ker\pi(y_i)$. We prove by induction on $k$ that $\HH_k\neq\{0\}$.
When $k=1$, this follows from Lemma \ref{lemma:main} applied to the operators $A=\sum_{i=2}^{n+1}\pi(y_i^*y_i)$
and $B=\pi(y_1)$. Since $\pi(y_{n+1})\neq 0$, it follows that $A\neq 0$, and hence $\ker(B)\neq\{0\}$.

Now assume that $\HH_{k-1}\neq\{0\}$. Let $A=\sum_{i=k+1}^{n+1}\pi(y_i^*y_i)$ and $B=\pi(y_k)$, and note that the operators $\pi(y_{n+1}),A,B,B^*$ map $\HH_{k-1}$ to itself, since $y_iy_j$ is a scalar multiple of $y_jy_i$ and $y_iy_j^*$ is a scalar multiple of $y_j^*y_i$ for all $i\neq j$. It follows from point (i) that $\pi(y_{n+1})|_{\HH_{k-1}}\neq 0$, so that $A$ is non-zero on $\HH_{k-1}$. The operator $A+B^*B$ restricts to the identity on $\HH_{k-1}$ and $[B,B^*]=(1-\mu)A$ with $\mu=q^2$ if $k<n$ and $\mu=q^4$ if $k=n$.
From Lemma \ref{lemma:main} applied to the restrictions of $A,B,B^*$ to $\HH_{k-1}$ it follows that $\ker(B)\cap\HH_{k-1}=\HH_k\neq \{0\}$.
\end{proof}

\begin{proof}[Proof of Prop.\ \ref{thm:two}]
Let $\pi$ be a bounded irreducible $*$-representation of $\A(\Sigma^{2n+1}_q)$ on a Hilbert space $\HH$ such that $\pi(y_{n+1})\neq 0$.
With an abuse of notation, we suppress the map $\pi$.
We know from Lemma \ref{lemma:atleast}(ii) that $V:=\bigcap_{i=1}^n\ker(y_i)$ is $\neq\{0\}$.
From the commutation relations we deduce that $y_{n+1}V\subset V$, so that $V$ carries a bounded $*$-representation of the commutative C*-algebra $C^*(y_{n+1},y_{n+1}^*)$ generated by $y_{n+1}$ and $y_{n+1}^*$.

Given $\mv{k}\in\N^n$ and $\xi\in V$ a unit vector, define:
$$
\ket{\mv{k}}_\xi:=\frac{1}{\sqrt{(q^2;q^2)_{k_1}\ldots(q^2;q^2)_{k_{n-1}}(q^4;q^4)_{k_n}}}(y_1^*)^{k_1}\ldots (y_n^*)^{k_n}\xi \;,
$$
where the $q$-shifted factorial is given by
$$
(a;b)_\ell:=\prod_{i=0}^{\ell-1}(1-ab^i) \;.
$$
Given $\mv{k}\in\mathbb{Z}^n$, set $\ket{\mv{k}}_\xi:=0$ if one of the components of $\mv{k}$ is negative.
From the commutation relations we deduce:
\begin{subequations}\label{eq:yiyn}
\begin{align}
y_i^*\ket{\mv{k}}_\xi &=q^{k_1+\ldots+k_{i-1}}\sqrt{1-q^{2k_i+2}}\ket{\mv{k}+\mv{e}_i}_\xi \qquad (i<n) \,,\\
y_n^*\ket{\mv{k}}_\xi &=q^{k_1+\ldots+k_{n-1}}\sqrt{1-q^{4k_n+4}}\ket{\mv{k}+\mv{e}_n}_\xi \;, \\
y_{n+1}\ket{\mv{k}}_\xi &=q^{|\mv{k}|+k_n}\ket{\mv{k}}_{y_{n+1}\xi} \;. \notag
\end{align}
\end{subequations}
If $W\subset V$ carries a subrepresentation of $C^*(y_{n+1},y_{n+1}^*)$, the Hilbert subspace of $\HH$ spanned by $\ket{\mv{k}}_\xi$ for $\xi\in W$ and $\mv{k}\in\N^n$ carries a subrepresentation of $\A(\Sigma^{2n+1}_q)$. Since $\HH$ is irreducible, $V$ carries an irreducible representation of $C^*(y_{n+1},y_{n+1}^*)$. This means that $V$ is one-dimensional. Let us fix a unit vector $\xi\in V$. Observe that the vectors
\begin{equation}\label{eq:set}
\{\ket{\mv{k}}_\xi\}_{\mv{k}\in\N^n}
\end{equation}
span $\HH$. Moreover $y_{n+1}\xi=\lambda\xi$ for some $\lambda\in\R$, and
$$
y_{n+1}\ket{\mv{k}}_\xi =\lambda\,q^{|\mv{k}|+k_n}\ket{\mv{k}}_{\xi} \;. 
$$
From now on, instead of $\ket{\mv{k}}_\xi $ we shall simply write $\ket{\mv{k}}$. By \eqref{eq:spherebis}, it follows that
$$
1=\inner{\mv{0}|\mv{0}}=\inner{\mv{0}\left|\sum\nolimits_{i=1}^{n+1}y_i^*y_i\right|\mv{0}}=|\lambda|^2 \;,
$$
hence $\lambda\in U(1)$. It remains to prove that the set \eqref{eq:set} is orthonormal, so that by adjunction from \eqref{eq:yiyn} we get the formulas in Prop.~\ref{prop:rep}.

Let $W_i$ be the span of vectors $\ket{\mv{k}}_\xi$ with $k_1=\ldots=k_i=0$.
It follows from the commutation relations that $y_k$ is zero on $W_i$ for all $k\leq i$. 

Applying the identities in Lemma \ref{lemma:aux} to a vector $\ket{\mv{k}}\in V_i$ we find that
\begin{align*}
y_i(y_i^*)^m
\ket{0,\ldots,0,k_{i+1},\ldots,k_n} &=(1-q^{2m})(y_i^*)^{m-1}\ket{0,\ldots,0,k_{i+1},\ldots,k_n} \\
y_n(y_n^*)^m\ket{\mv{0}} &=(1-q^{4m})(y_n^*)^{m-1}\ket{\mv{0}}
\end{align*}
for all $m\geq 1$ and all $1\leq i<n$.
Using \eqref{eq:yiyn} we find that
\begin{align*}
y_i\ket{0,\ldots,0,m-1,k_{i+1},\ldots,k_n} &=\sqrt{1-q^{2m}}\ket{0,\ldots,0,m-1,k_{i+1},\ldots,k_n} \\
y_n\ket{0,\ldots,0,m-1} &=\sqrt{1-q^{4m}}\ket{0,\ldots,0,m-1}
\end{align*}
Multiplying from the left by $\left<\mv{j}-\mv{e}_i\right|$ and using \eqref{eq:yiyn} again, we find that
\begin{align*}
\sqrt{1-q^{2j_i}}\inner{\mv{j}|\mv{k}}
  &=\sqrt{1-q^{2k_i}}\inner{\mv{j}-\mv{e}_i,\mv{k}-\mv{e}_i} \;, \\
\sqrt{1-q^{4j_n}}\inner{j_n\mv{e}_n|k_n\mv{e}_n}
  &=\sqrt{1-q^{4k_n}}\inner{(j_n-1)\mv{e}_n,(k_n-1)\mv{e}_n} \;,
\end{align*}
where the former is valid whenever $j_1,\ldots,j_i=k_1=\ldots=k_i=0$.
From these relations, an obvious induction proves that the set \eqref{eq:set} is orthonormal provided every vector
$\ket{\mv{k}}$ with $\mv{k}\neq\mv{0}$ is orthogonal to $\ket{\mv{0}}$.
But this is obvious. If $k_i\neq 0$ for some $i$, then
$$
\inner{\mv{k}|\mv{0}}\propto\inner{\mv{k}-\mv{e}_i|y_i|\mv{0}}=0
$$
since $y_i$ annihilates $\xi$.
\end{proof}

\end{document}